\documentclass[a4paper, 12pt]{amsart}

\usepackage[latin1]{inputenc}
\usepackage[T1]{fontenc}
\usepackage{amsmath}
\usepackage{amsfonts}
\usepackage{amssymb}
\usepackage{amscd}
\usepackage{calrsfs}
\usepackage[height=1.5em, width=2em, nohug, midshaft]{diagrams}
\usepackage[english]{babel}

\theoremstyle{plain}
\newtheorem{theorem}{Theorem}[section]
\newtheorem{lemma}[theorem]{Lemma}
\newtheorem{proposition}[theorem]{Proposition}
\newtheorem{corollary}[theorem]{Corollary}

\theoremstyle{definition}
\newtheorem{definition}[theorem]{Definition}
\newtheorem{remark}[theorem]{Remark}
\newtheorem{example}[theorem]{Example}

\newcommand{\st}{\ \vline\ } % 'such that' set notation bar

\newcommand{\tensor}{\otimes}

\newcommand{\Ltensor}{\overset{L}{\tensor}} % left derived tensor product
\newcommand{\iso}{\cong}      % isomorphism
\newcommand{\isoto}{\overset{\sim}{\to}}
\newcommand{\sheaf}[1]{\mathcal{#1}} % use script font for sheaves
\newcommand{\OO}{\sheaf{O}}   % structure sheaf of X is written \OO_X
\newcommand{\abdual}[1]{\widehat{#1}} % dual of an abelian variety
\newcommand{\dual}[1]{{#1}^\vee}
\newcommand{\res}[2]{\left.#1\right|_{#2}} % restriction bar
\newcommand{\dualres}[2]{\left.#1\right|_{#2}^{\vee}} % restriction bar + dualizing

\DeclareMathOperator{\Hom}{Hom}
\DeclareMathOperator{\Ext}{Ext}

\author{Martin G. Gulbrandsen}
\address{Royal Institute of Technology\\Stockholm\\Sweden}
\email{gulbr@kth.se}

\title[Fourier-Mukai transforms of line bundles]{Fourier-Mukai transforms of line bundles on derived equivalent abelian varieties}

\begin{document}

\maketitle
\tableofcontents

\section{Introduction}\label{sec:intro}

In connection with their work on generic vanishing and related
topics \cite{PP-GV}, Pareschi and Popa raise the
following question: Let $Y$ be a fine moduli space of stable
sheaves on a smooth projective variety $X$, and fix a universal
family $\sheaf{E}$ on $X\times Y$.  Let $\sheaf{L}$ be an ample
line bundle on $Y$, let $d=\dim Y$ and put
\begin{equation}\label{eq:G}
\sheaf{G}_n =
R^dp_{1*}\left(p_2^*(\sheaf{L}^{-n})\tensor\sheaf{E}\right)
\end{equation}
where $p_i$ denote the projections from $X\times Y$. Under
suitable hypotheses (for instance as in Example \ref{ex:IT}), the sheaf
$\sheaf{G}_n$ is locally free for $n$ sufficiently large;
it is the
Fourier-Mukai transform of $\sheaf{L}^{-n}$ with respect to
$\sheaf{E}$. What can be said about these bundles?
In particular, are they stable, and are they ample?

Not much is known about bundles of the above form --- in fact the
opposite construction has received much more attention,
i.e.~constructing sheaves on the moduli space $Y$ as
Fourier-Mukai transforms of sheaves on $X$.

In this text we work out the following special case:

\begin{theorem}\label{thm:main}
Let $X$ be an abelian variety and $Y$ a fine moduli space of
simple semihomogeneous bundles on $X$. Fix an ample line bundle
$\sheaf{L}$ on $Y$ and define $\sheaf{G}_n$ by \eqref{eq:G}.
Then, for sufficiently large $n$, the following holds:
\begin{enumerate}
\item $\sheaf{G}_n$ is a simple semihomogeneous bundle. In
particular it is stable.
\item $\sheaf{G}_n$ is ample if and only if the bundles
$\res{\sheaf{E}}{X\times\{y\}}$ parametrized by $Y$ are nef.
\end{enumerate}
\end{theorem}

The definition of semihomogeneous bundles is recalled in Section
\ref{sec:terminology}. Our viewpoint is that their moduli spaces
are of the simplest possible form: In the mostly expository sections
\ref{sec:moduli} and \ref{sec:derived} we recall results of
Mukai and Orlov showing that any moduli space $Y$ of
semihomogeneous bundles on $X$ is again an abelian variety of the
same dimension as $X$, and the Fourier-Mukai functor associated to
a universal family is an equivalence $D(X)\iso D(Y)$ of derived
categories. Suitably formulated converses to these statements
also hold.

Any line bundle is semihomogeneous, and the prototype for a
moduli space of semihomogeneous sheaves is the dual abelian
variety $Y=\abdual{X}$, equipped with the normalized Poincar\'e
bundle $\sheaf{P}$. In this case, the above theorem is well
known; in fact the bundle $\sheaf{G}_n$ is stable and ample
already for $n=1$. This can be deduced from a result of Mukai
\cite{mukai81},
saying that the pullback of $\sheaf{G}_1$ under the canonical
isogeny
\begin{equation*}
\phi_{\sheaf{L}}\colon \abdual{X}\to X,\quad \xi\mapsto
T_{\xi}^*(\sheaf{L})\tensor\dual{\sheaf{L}}
\end{equation*}
(viewing $X$ as the dual of $\widehat{X}$), is just
\begin{equation*}
H^0(\abdual{X}, \sheaf{L}) \tensor_k \sheaf{L},
\end{equation*}
i.e.~a direct sum of a suitable number of copies of $\sheaf{L}$
itself.

One motivation for studying bundles of the form $\sheaf{G}_n$ is
the r\^ole they play in Pareschi and Popa's approach to generic
vanishing: Returning to the general case, with $Y$ an arbitrary
moduli space of sheaves on a smooth projective variety $X$, we
say that a sheaf $\sheaf{F}$ on $X$ satisfies generic vanishing
with respect to the universal family $\sheaf{E}$ if, for each
$i$, the closed set
\begin{equation*}
\left\{
y\in Y \st H^i(X, \sheaf{F}\tensor\sheaf{E}_y) \ne 0
\right\}
\end{equation*}
has codimension at least $i$. A (special case of a) criterion of Pareschi and Popa
says that the bundle $\sheaf{G}_n$ can be used to test generic
vanishing. Namely, $\sheaf{F}$ satisfies generic vanishing if and
only if
\begin{equation*}
H^i(X, \sheaf{F}\tensor\sheaf{G}_n) = 0
\end{equation*}
for all $i>0$. Here, it is enough to test with a bundle
$\sheaf{G}_n$ associated to a fixed ample line bundle $\sheaf{L}$
and a fixed, but large, integer $n$.
Pareschi and Popa use their criterion, together with Mukai's
description of $\sheaf{G}_n$ in the case of an abelian variety
and its dual, to prove a generalization of the Green-Lazarsfeld
generic vanishing theorem \cite[Theorem A]{PP-GV}. The upshot is
that a good understanding of the bundles $\sheaf{G}_n$, and in
particular their positivity properties, is required to make the
generic vanishing criterion effective.

The first part of the theorem (with weaker hypotheses) is obtained
as Corollary \ref{cor:stable}, as an immediate consequence of the results
of Mukai and Orlov. Theorem \ref{thm:index} is a more precise version of
the second part: 
We prove that the bundle
$\sheaf{G}_n$ always satisfies an index theorem, and its index can
be computed. By demanding its index to be zero, we arrive at the
criterion for ampleness, stated as Corollary \ref{cor:ample}.

This project was initiated during the summer school Pragmatic 2007
in Catania, Italy. I am most grateful to Mihnea Popa and Giuseppe
Pareschi, not only for their lectures, but also for putting an
immense effort into helping each participant identify and develop
concrete and accessible problems.

\section{Terminology}\label{sec:terminology}

Throughout, let $X$ denote an abelian variety of dimension $g$ over an
algebraically closed field $k$ of arbitrary characteristic. We write
$\abdual{X}$ for the dual abelian variety.

Points in $\abdual{X}$ are denoted by Greek letters $\xi,\zeta,\dots$,
and we use the same symbols for the corresponding homogeneous line
bundles on $X$. For each point $x\in X$, we write $T_x\colon X\to X$
for translation by $x$. If $\sheaf{L}$ is a line bundle on $X$,
we write $K(\sheaf{L})\subseteq X$ for the subgroup of points
$x\in X$ satisfying $T_x^*(\sheaf{L})\iso\sheaf{L}$.

We use the words vector bundle and line bundle as synonyms for locally
free sheaf and invertible sheaf. By stability, we mean
Gieseker-stability with respect to any fixed polarization, the
choice of which will not matter to us.

\begin{definition}[Mukai \cite{mukai78}]
A coherent sheaf $\sheaf{E}$ on $X$ is \emph{semihomogeneous} if the
locus
\begin{equation*}
\Gamma(\sheaf{E}) = \{ (x,\xi) \in X\times\abdual{X} \st
T_x^*(\sheaf{E})\iso\sheaf{E}\tensor\xi \}
\end{equation*}
has dimension $g$.
\end{definition}

If $\sheaf{E}$ locally free, then it is semihomogeneous if and only
if the following condition holds: For each $x\in X$, there is a
$\xi\in\abdual{X}$, such that
\begin{equation*}
T_x^*(\sheaf{E})\iso\sheaf{E}\tensor\xi.
\end{equation*}
The equivalence with the
definition given above follows from noting that the kernel of the
first projection $p_1\colon \Gamma(\sheaf{E})\to X$ is finite: In fact,
its kernel is contained in the group of $r$-torsion points
$\abdual{X}_r$, where $r$ is the rank of $\sheaf{E}$.

Next we set up notation for the Fourier-Mukai transform. We write
$D(X)$ for the bounded derived category of a variety $X$, equipped with the
auto\-functors $\sheaf{C}\mapsto\sheaf{C}[i]$ that shift a complex
$\sheaf{C}$ the specified number $i$ steps to the left. We view
a sheaf as a complex concentrated in degree $0$; thus $D(X)$
contains the category of coherent sheaves.
Let $X$ and
$Y$ be two varieties (both will be abelian varieties in our context),
and let
\begin{equation*}
X \xleftarrow{p_1} X\times Y \xrightarrow{p_2} Y
\end{equation*}
denote the projections. To any coherent sheaf $\sheaf{E}$ (or, more
generally, any bounded complex) on the product $X\times Y$, we associate
a pair of functors between the derived categories of $X$ and $Y$:

\begin{definition}
The \emph{Fourier-Mukai functors} with kernel $\sheaf{E}$ are the two
functors
\begin{align*}
\Phi_{\sheaf{E}}&\colon D(X) \to D(Y),
& \Phi_{\sheaf{E}}(-) &= Rp_{2*}(p_1^*(-)\Ltensor\sheaf{E})\\
\Psi_{\sheaf{E}}&\colon D(Y) \to D(X),
& \Psi_{\sheaf{E}}(-) &= Rp_{1*}(p_2^*(-)\Ltensor\sheaf{E}).
\end{align*}
We write $\Phi_{\sheaf{E}}^i(-)$ and $\Psi_{\sheaf{E}}^i(-)$ for the $i$'th
cohomology sheaf of $\Phi_{\sheaf{E}}(-)$ and $\Psi_{\sheaf{E}}(-)$.
\end{definition}

\begin{definition}
Given a triple $(X,Y,\sheaf{E})$ as above and a coherent sheaf
$\sheaf{F}$ on $X$, we say that
\begin{enumerate}
\item $\sheaf{F}$ satisfies the \emph{index theorem} (IT) with respect to
  $\sheaf{E}$ if there exists an integer $i_0$ such that
\begin{equation*}
H^i(X, \sheaf{F}\tensor\sheaf{E}_y) = 0\quad\text{for all $i\ne i_0$}.
\end{equation*}
\item $\sheaf{F}$ satisfies the \emph{weak index theorem} (WIT) with respect to
  $\sheaf{E}$ if there exists an integer $i_0$ such that
\begin{equation*}
\Phi^i_{\sheaf{E}}(\sheaf{F}) = 0\quad\text{for all $i\ne i_0$}.
\end{equation*}
\end{enumerate}
\end{definition}

Suppose that the kernel $\sheaf{E}$ of the Fourier-Mukai functor is
a $Y$-flat coherent sheaf. Then the base change theorem in cohomology
shows that IT implies WIT. The integer $i_0$ in the definition will be
referred to as the $\sheaf{E}$-index of $\sheaf{F}$, denoted
$i_{\sheaf{E}}(\sheaf{F})$.

\begin{example}\label{ex:IT}
Let $X$ and $Y$ be projective varieties, and
let $\sheaf{E}$ be a vector bundle on $X\times Y$. Assume that
$Y$ has a dualizing sheaf $\omega_Y$.
If $\sheaf{L}$ is an ample line bundle on $Y$ and $n$
is sufficiently large, then $\sheaf{L}^{-n}$ satisfies IT with
respect to $\sheaf{E}$, and its $\sheaf{E}$-index is $d=\dim Y$.
This follows
from Serre's theorems: We have
\begin{equation*}
  H^i(Y, \sheaf{L}^{-n}\tensor \res{\sheaf{E}}{\{x\}\times Y}) \iso
  \dual{H^{d-i}(Y, \sheaf{L}^n\tensor
    \dualres{\sheaf{E}}{\{x\}\times Y}\tensor\omega_Y)}
\end{equation*}
and the latter vanishes for $n$ sufficiently large if $i\ne d$.
The bound on $n$ can be made independent of $x$, since the
dimension of the cohomology vector space above is upper semicontinuous
as a function of $x$.
\end{example}

\begin{definition}
Let $\sheaf{F}$ be a coherent sheaf satisfying WIT with respect to
$\sheaf{E}$, and let $i_0$ denote its $\sheaf{E}$-index. The
\emph{Fourier-Mukai transform} of $\sheaf{F}$ with respect to
$\sheaf{E}$ is the coherent sheaf
$\Phi^{i_0}_{\sheaf{E}}(\sheaf{F})$.
\end{definition}

When we do not specify the kernel explicitly, we will always mean the
Fourier-Mukai functor with respect to the Poincar\'e line bundle on
$X\times\abdual{X}$. Thus, in this case, and in this case only, we
will write $\Phi$ and $\Psi$ with no subscript, and, if $\sheaf{F}$
satisfies WIT, its index (i.e.~its $\sheaf{P}$-index) is denoted
$i(\sheaf{F})$. In this case we also use the notation
\begin{equation*}
\abdual{\sheaf{F}} = \Phi^{i(\sheaf{F})}(\sheaf{F})
\end{equation*}
for the Fourier-Mukai transform.

\section{Moduli spaces of semihomogeneous
bundles}\label{sec:moduli}

Let $M$ be the moduli space of stable sheaves on $X$. Let $\sheaf{E}$
be a stable vector bundle of rank $r$ on $X$, and let $Y\subseteq M$
be the connected component containing $\sheaf{E}$. Recall that
the tangent space to $Y$ at $\sheaf{E}$ is canonically isomorphic
to $\Ext^1_X(\sheaf{E}, \sheaf{E})$.

\begin{proposition}\label{prop:semihom}
The following are equivalent.
\begin{enumerate}
\item $\sheaf{E}$ is semihomogeneous
\item $\dim_k \Ext^1_X(\sheaf{E},\sheaf{E}) = g$
\item $Y$ is an abelian variety isogeneous to $X$.
\end{enumerate}
\end{proposition}

\begin{proof}
The equivalence of (1) and (2) is due to Mukai \cite{mukai78}, and this
is the deepest part.
It is obvious that (3) implies (2). We next show that (1)
implies (3).

Let $\sheaf{Q}$ be the determinant of $\sheaf{E}$. We can form a
commutative diagram
\begin{diagram}
&\abdual{X} &\rTo^\tau &Y \\
& &\rdTo_{r_{\abdual{X}}} & \dTo_\delta\\
& & & \abdual{X}
\end{diagram}
where the twisting map $\tau$ sends a homogeneous line bundle
$\xi\in\abdual{X}$ to $\sheaf{E}\tensor\xi$ and the determinant map
$\delta$ sends a sheaf $\sheaf{F}\in Y$ to the homogeneous line bundle
$\det(\sheaf{F})\tensor\sheaf{Q}^{-1}$. The composition is
multiplication by $r$, since $\det(\sheaf{E}\tensor\xi)\iso
\sheaf{Q}\tensor\xi^r$.

Since the composed map $r_{\abdual{X}}$ is finite, so is $\tau$, and
hence its image is $g$-dimensional. The
bundles parametrized by the image of $\tau$ are clearly
semihomogeneous, hence, using the equivalence (1)$\iff$(2), we conclude that $Y$ is
nonsingular and $g$-dimensional at all points of
$\tau(\abdual{X})$. Since $Y$ is connected this implies that
$\tau(\abdual{X})=Y$, and so $Y$ is a nonsingular $g$-dimensional
variety.

Now let $Y(\sheaf{Q})\subset Y$ be the subscheme
$\delta^{-1}(0)$, i.e.~the locus in $Y$ parametrizing bundles with
fixed determinant $\sheaf{Q}$. Then there is a Cartesian diagram
\begin{diagram}
&\abdual{X}\times Y(\sheaf{Q}) &\rTo &Y \\
&\dTo & &\dTo_\delta \\
&\abdual{X} &\rTo^{r_{\abdual{X}}} & \abdual{X}
\end{diagram}
where the left map is projection onto the first factor and the top map
sends a pair $(\xi,\sheaf{F})$ to the tensor product $\sheaf{F}\tensor\xi$. In
particular, the determinant map $\delta$ is locally trivial in the
étale topology. Since $Y$ is nonsingular, this implies that $\delta$
is étale. A variety admitting an étale map to an abelian variety is
itself an abelian variety, so we are done.
\end{proof}

\section{Derived equivalent abelian varieties as moduli
spaces}\label{sec:derived}

Let $X$ and $Y$ be abelian varieties, and suppose there exists a
derived equivalence
\begin{equation*}
D(X) \isoto D(Y).
\end{equation*}
By results of Orlov \cite{orlov}, any such equivalence is a
Fourier-Mukai transform $\Phi_{\sheaf{E}}$, with kernel a sheaf
$\sheaf{E}$ (i.e.~a complex concentrated in one degree, although
not necessarily degree zero) on the product
$X\times Y$.  Moreover, this sheaf is semihomogeneous. The
semihomogeneity is perhaps only almost explicit in Orlov's work --- so
here is a short account:

Associated to $\sheaf{E}$, Orlov constructs an isomorphism
\begin{equation*}
f\colon X\times\abdual{X} \isoto Y\times\abdual{Y},
\end{equation*}
which on points is given by
\begin{gather*}
f(x,\xi) = (y,\zeta)\label{eq:graph}\\
\Updownarrow\\
T_{(x,y)}^*(\sheaf{E}) \iso \sheaf{E}\tensor (p_1^*(\xi^{-1})\tensor
p_2^*(\zeta)).\label{eq:phi}
\end{gather*}
This says that a quadruple $(x,\xi,y,\zeta)$ belongs to the graph of
$f$ if and only if $(x,y,\xi^{-1},\zeta)$ belongs to
$\Gamma(\sheaf{E})$, with notation as in Section
\ref{sec:terminology}. Since
the graph of $f$ is $2g$-dimensional, so is $\Gamma(\sheaf{E})$, which
shows that $\sheaf{E}$ is semihomogeneous.

For simplicity, we will assume that $\sheaf{E}$ is also locally
free.

\begin{proposition}
Let $X$ and $Y$ be abelian varieties and $\sheaf{E}$ a vector bundle
on their product $X\times Y$. Then the following are equivalent.
\begin{enumerate}
\item The Fourier-Mukai transform $\Phi_{\sheaf{E}}\colon D(X)\to D(Y)$ with
kernel $\sheaf{E}$ is an equivalence.
\item The variety $Y$, equipped with the family $\sheaf{E}$, is a
fine moduli space of simple semihomogeneous vector bundles on $X$.
\end{enumerate}
\end{proposition}

\begin{remark}\label{rem:stable}
Every simple semihomogeneous vector bundle is (Gieseker) stable,
with respect to any polarization \cite{mukai78}.
\end{remark}

\begin{proof}
By a criterion of Bondal and Orlov \cite[Corollary 7.5]{huybrechts}, the
functor $\Phi_{\sheaf{E}}$ is fully faithful if and only if
\begin{align}
\Hom(\sheaf{E}_y,\sheaf{E}_y)&=k & &\text{for all $y$}\label{eq:simple}\\
\Ext^i(\sheaf{E}_y, \sheaf{E}_{y'})&=0 & &\text{for all $i$ and $y\ne
  y'$}.\label{eq:orthogonal}
\end{align}
We also need the fact that, by the triviality of the canonical bundles on $X$
and $Y$, the functor $\Phi_{\sheaf{E}}$ is fully faithful if and only if it is an
equivalence \cite[Corollary 7.8]{huybrechts}.

First assume that $Y$, equipped with $\sheaf{E}$, is a moduli space of
simple semihomogeneous vector bundles on $X$. Then
\eqref{eq:simple} is fulfilled since the fibres $\sheaf{E}_y$ are
simple, and \eqref{eq:orthogonal} follows from Mukai's work on
homogeneous and semihomogeneous vector bundles --- see Lemma~4.8 in Orlov's
paper \cite{orlov}. So $\Phi_{\sheaf{E}}$ is fully faithful and hence an
equivalence.

Conversely, assume $\Phi_{\sheaf{E}}$ is an equivalence. Since
\eqref{eq:simple} is fulfilled, the fibres $\sheaf{E}_y$ are simple,
and they are semihomogeneous since $\sheaf{E}$ is a semihomogeneous
bundle on $X\times Y$. Since $\sheaf{E}$ is locally free, it is flat
over $Y$, and so induces a morphism
\begin{equation*}
f\colon Y \to M
\end{equation*}
to the moduli space $M$ of stable sheaves on $X$. By Proposition
\ref{prop:semihom}, this map $f$ hits a component $M'$ of $M$ which is an abelian
$g$-dimensional variety. Since \eqref{eq:orthogonal} is fulfilled,
all distinct fibres $\sheaf{E}_y$ and $\sheaf{E}_{y'}$ are non
isomorphic, which says that $f$ has degree $1$. Thus $f$ is an
isomorphism, being a degree $1$ map
between abelian varieties of the same dimension.
\end{proof}

\begin{corollary}\label{cor:stable}
Let $Y$ be a fine moduli space for simple semihomogeneous vector bundles on
$X$, with a fixed universal family $\sheaf{E}$. Let $\sheaf{L}$ be a
line bundle on $Y$ satisfying IT with respect to $\sheaf{E}$. Then the
Fourier-Mukai transform $\sheaf{G}$ of $\sheaf{L}$ with respect to $\sheaf{E}$
is a simple semihomogeneous vector bundle on $X$. In particular it is
stable.
\end{corollary}

\begin{proof}
By the proposition, the functor $\Phi_{\sheaf{E}}$ is an equivalence
of categories, which implies that $\Psi_{\sheaf{E}}$ is an equivalence
also. Thus
\begin{equation*}
\Ext^i_{Y}(\sheaf{L},\sheaf{L}) \iso \Ext^i_{X}(\sheaf{G},\sheaf{G})
\end{equation*}
for all $i$.
In particular, for $i=0$ we get that $\sheaf{G}$ is
simple, and for $i=1$ we get $\dim
\Ext^1_X(\sheaf{G},\sheaf{G})=g$, which implies $\sheaf{G}$ is
semihomogeneous, by Proposition \ref{prop:semihom}. By Remark
\ref{rem:stable}, it is also stable.
\end{proof}

The first part of Theorem \ref{thm:main} follows, since the very
negative line bundle $\sheaf{L}^{-n}$ considered there satisfies
IT, by Example \ref{ex:IT}.

\section{Index theorems}\label{sec:index}

As is well known, every nondegenerate line bundle $\sheaf{L}$ satisfies IT
(with respect to the Poincar\'e bundle). In this section we show that
degenerate line bundles satisfy WIT, and we extend these results to
simple semihomogeneous vector bundles.

The starting point is the following construction by Kempf
\cite{kempf}: Let
$\sheaf{L}$ be a degenerate line bundle on $X$. Let $Y\subseteq
X$ be the identity component of $K(\sheaf{L})$,
with reduced structure, and
let
\begin{equation*}
\pi\colon X\to X/Y
\end{equation*}
be the quotient. Then there exist a nondegenerate line bundle
$\sheaf{M}$ on $X/Y$ and a homogeneous line bundle $\xi\in\abdual{X}$
such that
\begin{equation}\label{eq:degenerate}
\sheaf{L}\iso\pi^*(\sheaf{M})\tensor\xi.
\end{equation}

\begin{proposition}\label{prop:degenerate}
Let $\sheaf{L}$ be a degenerate line bundle, and write
$\sheaf{L} = \pi^*(\sheaf{M})\tensor\xi$ with $\sheaf{M}$ nondegenerate,
as above. Then $\sheaf{L}$ satisfies WIT with index
\begin{equation*}
i(\sheaf{L}) = \dim K(\sheaf{L}) + i(M)
\end{equation*}
and its Fourier-Mukai transform is
\begin{equation*}
\abdual{\sheaf{L}} \iso
T_{\xi}^*(\abdual{\pi}_*(\abdual{\sheaf{M}})).
\end{equation*}
\end{proposition}

\begin{remark}
Note that, since $\abdual{\pi}$ is an embedding, the expression
$\abdual{\pi}_*(\abdual{\sheaf{M}})$ above simply means
$\abdual{\sheaf{M}}$ considered as a torsion sheaf on $\abdual{X}$.
\end{remark}

\begin{proof}
For any homogeneous line bundle $\xi$, there is a natural isomorphism
\cite{mukai81} 
\begin{equation*}
\Phi((-)\tensor\xi) \iso T_{\xi}^*(\Phi(-)).
\end{equation*}
Furthermore, for an arbitrary homomorphism $\phi\colon X\to Z$ of abelian
varieties, there is an isomorphism \cite[Section 11.3]{polishchuk}
\begin{equation*}
\Phi\circ L\phi^*[d] \iso R\abdual{\phi}_*\circ\Phi
\end{equation*}
where $d=\dim X-\dim Z$. Note that $\Phi$ on the left hand side is the
Fourier-Mukai functor with kernel the Poincar\'e bundle on
$X\times\abdual{X}$, whereas the same symbol on the right hand side is
the Fourier-Mukai functor with kernel the Poincar\'e bundle on
$Z\times\abdual{Z}$.

Since $\pi$ is flat and $\abdual{\pi}$ finite, we have  $L\pi^* = \pi^*$ and
$R\abdual{\pi}_*=\abdual{\pi}_*$. Thus we get, for every integer $i$,
\begin{align*}
\Phi^i(\pi^*(\sheaf{M})\tensor\xi)
&\iso T_{\xi}^*(\Phi^i(\pi^*(\sheaf{M})))\\
&\iso T_{\xi}^*(\abdual{\pi}_*(\Phi^{i-d}(\sheaf{M}))),
\end{align*}
where $d=\dim Y$, which is also the dimension of $K(\sheaf{L})$.
Since $\sheaf{M}$ is nondegenerate, it satisfies IT
with some index $i(\sheaf{M})$. The claim follows.
\end{proof}

\begin{corollary}\label{cor:index}
Every line bundle $\sheaf{L}$ satisfies WIT. The index of a line bundle
satisfies
\begin{enumerate}
\item $i(\dual{\sheaf{L}}) = g + \dim K(\sheaf{L}) - i(\sheaf{L})$,
\item $i(\sheaf{L}\tensor\xi) = i(\sheaf{L})$ for all
  $\xi\in\abdual{X}$,
\item $i(\sheaf{L}^n) = i(\sheaf{L})$ for all $n>0$.
\end{enumerate}
\end{corollary}

\begin{proof}
When $\sheaf{L}$ is nondegenerate, these are standard
facts. Otherwise, write $\sheaf{L}$ as $\pi^*(\sheaf{M})\tensor\xi$
as before, and apply the nondegenerate case to $\sheaf{M}$. The
claims then follow from the formula for $i(\sheaf{L})$ in the
proposition, since
$\dim K(\sheaf{L})$ is invariant under dualizing, twisting with
homogeneous line bundles and positive tensor powers.
\end{proof}

As a further application of Proposition \ref{prop:degenerate}, we give a
characterization of nef line bundles on abelian varieties.

\begin{corollary}\label{cor:nef}
  The following are equivalent conditions on a line bundle $\sheaf{L}$.
  \begin{enumerate}
  \item $\sheaf{L}$ is nef
  \item $i(\sheaf{L}) = \dim K(\sheaf{L})$
  \item There exist an abelian subvariety $Y\subseteq X$ and an ample
    line bundle on $X/Y$ such that
    \begin{equation*}
      \sheaf{L} \iso \pi^*(\sheaf{M})\tensor\xi
    \end{equation*}
    where $\pi\colon X\to X/Y$ is the quotient and $\xi\in\abdual{X}$
    is a homogeneous line bundle.
  \end{enumerate}
\end{corollary}

\begin{remark}
In particular, on a simple abelian variety, a nef line bundle is
either ample or algebraically equivalent to $\OO_X$.
\end{remark}

\begin{proof}
  Let $Y\subseteq X$ be the identity component of $K(\sheaf{L})$, with
  reduced structure (thus $Y=0$ if $\sheaf{L}$ is nondegenerate).
  Then, with no assumptions on $\sheaf{L}$, we can write
  \begin{equation*}
    \sheaf{L} \iso \pi^*(\sheaf{M})\tensor\xi
  \end{equation*}
  for a nondegenerate line bundle $\sheaf{M}$ on $X/Y$, by the results
  of Kempf. By Proposition
  \ref{prop:degenerate}, condition (2) holds if and only if $\sheaf{M}$ has
  index zero. Since a nondegenerate line bundle has index zero if and
  only if it is ample, this proves the equivalence of (2) and (3).
  
  Next we prove that (1) implies (3). With $Y$ and $\sheaf{M}$ as
  before, assume $\pi^*(\sheaf{M})\tensor\xi$ is nef. Then
  $\pi^*(\sheaf{M})$ is also nef. But then $\sheaf{M}$ was nef to
  begin with \cite[Example 1.4.4]{lazarsfeld}. Since $\sheaf{M}$ is also
  nondegenerate, it is ample \cite[Corollary 1.5.18]{lazarsfeld}.

  Finally, it is clear that (3) implies (1).
\end{proof}

\section{Semihomogeneous vector bundles}\label{sec:semihom}

Let $\sheaf{E}$ be a simple semihomogeneous vector bundle on $X$. Recall the
following:
\begin{enumerate}
\item There is an isogeny $f\colon Y\to X$ and a line bundle
  $\sheaf{L}$ on $Y$ such that $f^*(\sheaf{E}) \iso \sheaf{L}^{\oplus
    r}$.
\item There is an isogeny $f\colon Y\to X$ and a line bundle
  $\sheaf{L}$ on $Y$ such that $\sheaf{E}\iso f_*(\sheaf{L})$.
\end{enumerate}
In fact, among simple vector bundles, the semihomogeneous ones are
characterized by either of these two properties \cite{mukai78}. In this
section we will use these facts to reduce many questions about
$\sheaf{E}$ to the corresponding questions about its determinant line
bundle $\sheaf{Q}$.

The following definition extends standard terminology for line bundles
on abelian varieties to simple semihomogeneous vector bundles.

\begin{definition}
A simple semihomogeneous vector bundle $\sheaf{E}$ is
\emph{degenerate} if its Euler characteristic $\chi(\sheaf{E})$
is zero. Otherwise it is \emph{nondegenerate}.
\end{definition}

\begin{proposition}\label{prop:det}
Let $\sheaf{E}$ be a simple semihomogeneous vector bundle on $X$, with
determinant line bundle $\sheaf{Q}$.
\begin{enumerate}
\item $\sheaf{E}$ is nondegenerate if and only if $\sheaf{Q}$ is nondegenerate.
\item $\sheaf{E}$ is ample if and only if $\sheaf{Q}$ is ample.
\item $\sheaf{E}$ is nef if and only if $\sheaf{Q}$ is nef.
\end{enumerate}
\end{proposition}

\begin{proof}
Part (1) follows from Mukai's formula \cite{mukai78}
\begin{equation*}
\chi(\sheaf{E}) = \chi(\sheaf{Q})/r^{g-1}
\end{equation*}
for the Euler characteristic of $\sheaf{E}$, where $r$ is its rank.

To prove (2) and (3), choose an isogeny $f\colon Y\to X$ such that
$f^*(\sheaf{E})\iso\sheaf{L}^{\oplus r}$ for a line bundle $\sheaf{L}$
on $Y$. Then $f^*(\sheaf{Q})\iso\sheaf{L}^r$. Since $f$ is a finite
map, a vector bundle on $X$ is ample (nef) if and only if its
pullback to $Y$ is. Thus we have
\begin{align*}
\sheaf{E} \text{ ample (nef)} &\iff \sheaf{L}^{\oplus r} \text{ ample (nef)}\\
\sheaf{Q} \text{ ample (nef)} &\iff \sheaf{L}^r \text{ ample (nef)}
\end{align*}
and the two statements on the right are both equivalent to the
ampleness (nefness) of $\sheaf{L}$.
\end{proof}

\begin{proposition}\label{prop:wit}
  Let $\sheaf{E}$ be a simple semihomogeneous vector bundle. Then
  $\sheaf{E}$ satisfies WIT, and its index equals the index of its
  determinant line bundle. If $\sheaf{E}$ is nondegenerate, then it
  satisfies IT.
\end{proposition}

\begin{proof}
Let 
\begin{equation*}
f\colon Y\to X
\end{equation*}
be an isogeny such that $f_*(\sheaf{L})\iso\sheaf{E}$ for some line
bundle $\sheaf{L}$ on $Y$. Then $\sheaf{L}$ satisfies
WIT by Corollary \ref{cor:index}. Furthermore we have
\cite[Section 3]{mukai81}
\begin{equation*}
% the \left. and \right. below is to lift the star * above the
% hat ^
\Phi^i(f_*(\sheaf{L})) \iso \left.\abdual{f}\right.^*(\Phi^i(\sheaf{L})),
\end{equation*}
which implies that $\sheaf{E}$ satisfies WIT, and its index equals the
index of $\sheaf{L}$.

Next we show that the index of $\sheaf{L}$ equals the index of the
determinant $\sheaf{Q}$ of $\sheaf{E}$.
Firstly, we have \cite[Lemma
  6.21]{mukai78}
\begin{equation}\label{eq:detpullback}
f^*(\sheaf{Q}) \sim \sheaf{L}^d\quad\text{(algebraic equivalence)}
\end{equation}
where $d$ is the degree of $f$. By Corollary \ref{cor:index}, it follows
that $\sheaf{L}$ and $f^*(\sheaf{Q})$ have the same
index. Moreover \cite[Section 3]{mukai81}, there are isomorphisms
\begin{equation*}
\Phi^i(f^*(\sheaf{Q})) \iso \abdual{f}_*(\Phi^i(\sheaf{Q}))
\end{equation*}
for all $i$, so $f^*(\sheaf{Q})$ has the same index as
$\sheaf{Q}$. This proves the first part.

Since the dual map $\abdual{f}$ is again
an isogeny, we have that every $\zeta\in\abdual{Y}$ is a pullback
$f^*(\xi)$ of some element $\xi\in\abdual{X}$. By the projection
formula,
\begin{equation*}
H^i(Y, \sheaf{L}\tensor\zeta)
\iso
H^i(X, f_*(\sheaf{L}\tensor\zeta))
\iso
H^i(X, \sheaf{E}\tensor\xi).
\end{equation*}
Thus $\sheaf{L}$ satisfies IT if and only if $\sheaf{E}$
does. But, if $\sheaf{E}$ is nondegenerate, then so is $\sheaf{Q}$ by
Proposition \ref{prop:det}, and then \eqref{eq:detpullback} shows that
\begin{equation*}
d^g\chi(\sheaf{L}) = \chi(\sheaf{L}^d) = d \chi(\sheaf{Q}) \ne 0,
\end{equation*}
so $\sheaf{L}$ is nondegenerate also. Thus $\sheaf{L}$ satisfies IT. By
what we just said, this proves that $\sheaf{E}$ satisfies IT.
\end{proof}

\section{Fourier-Mukai transforms of negative line
bundles}\label{sec:negative}

\begin{lemma}\label{lem:push}
Let $Y$ be a fine moduli space of simple semihomogeneous vector bundles on
$X$ and let $\sheaf{E}$ be a fixed universal family on $X\times Y$.
\begin{enumerate}
\item There exists an isogeny $f\colon\abdual{X}\to Y$, that sends
  an element $\sheaf{\xi}\in\abdual{X}$ to the bundle
  $\res{\sheaf{E}}{X\times\{0\}}\tensor\xi$.
\item For every $i$, the sheaf $R^ip_{2*}(\sheaf{E})$ is zero if and
  only if $\Phi^i(\res{\sheaf{E}}{X\times\{0\}})$ is zero.
\end{enumerate}
\end{lemma}

\begin{proof}
The bundle
\begin{equation*}
\sheaf{P}\tensor p_1^*(\res{\sheaf{E}}{X\times\{0\}})
\end{equation*}
on $X\times\abdual{X}$, considered as a family over $\abdual{X}$ of simple
semihomogeneous bundles, induces a map $f\colon \abdual{X}\to Y$. This
is the map required in the first part.

Moreover, there exists a line bundle
$\sheaf{L}$ on $\abdual{X}$ such that
\begin{equation*}
(1_X\times f)^*(\sheaf{E}) \iso
p_2^*(\sheaf{L})\tensor\sheaf{P}\tensor p_1^*(\res{\sheaf{E}}{X\times\{0\}}).
\end{equation*}
Hence, using flat base change for higher push forward and the
projection formula, we get
\begin{align*}
f^*(R^ip_{2*}(\sheaf{E}))
&\iso
R^ip_{2*}((1_X\times f)^*(\sheaf{E}))\\
&\iso
R^ip_{2*}(p_2^*(\sheaf{L})\tensor\sheaf{P}\tensor
p_1^*(\res{\sheaf{E}}{X\times\{0\}}))\\
&\iso
\sheaf{L}\tensor R^ip_{2*}(p_1^*(\res{\sheaf{E}}{X\times\{0\}})\tensor\sheaf{P})\\
&=
\sheaf{L}\tensor \Phi^i(\res{\sheaf{E}}{X\times\{0\}}).
\end{align*}
The second part of the lemma follows.
\end{proof}

\begin{theorem}\label{thm:index}
  Let $Y$ be a fine moduli space of simple semihomogeneous vector
  bundles on $X$, and let $\sheaf{E}$ be a fixed universal family on
  $X\times Y$. Also let $\sheaf{L}$ be an ample line bundle on $Y$.
  There exists
  an integer $n_0$ such that for all $n\ge n_0$,
  \begin{enumerate}
  \item the line bundle $\sheaf{L}^{-n}$ satisfies the index theorem with
    respect to $\sheaf{E}$, and its index is $g$;
  \item the Fourier-Mukai transform $\sheaf{G}_n =
    \Psi^g_{\sheaf{E}}(\sheaf{L}^{-n})$ with respect to $\sheaf{E}$ is
    nondegenerate and its index is
    \begin{equation*}
      i(\sheaf{G}_n) = i(\sheaf{Q}) - \dim K(\sheaf{Q}) 
    \end{equation*}
    where $\sheaf{Q}$ is the determinant line bundle of any of the
    bundles $\res{\sheaf{E}}{X\times\{y\}}$ parametrized by $Y$.
  \end{enumerate}
\end{theorem}

\begin{proof}
  The first part was established in Example \ref{ex:IT}.
  For the second part, we make use of the Fourier-Mukai equivalence
  induced by $\sheaf{E}$. As $\sheaf{L}^{-n}$ has index $g$, we have
  \begin{equation}\label{eq:fm-trans1}
    \Psi_{\sheaf{E}}(\sheaf{L}^{-n})[g] \iso \sheaf{G}_n
  \end{equation}
  (where, as usual, the bundle on the right hand side is considered as
  a complex concentrated in degree zero). Furthermore, by Proposition
  \ref{prop:wit}, the semihomogeneous vector bundle
  $\dualres{\sheaf{E}}{X\times\{0\}}$ satisfies WIT with index equal
  to the index of $\dual{\sheaf{Q}}$, which is
  \begin{equation*}
    i_0 = g + \dim K(\sheaf{Q}) - i(\sheaf{Q})
  \end{equation*}
  by Corollary \ref{cor:index}. Now apply Lemma \ref{lem:push} to
  $\dual{\sheaf{E}}$ to conclude that $\Phi^i_{\dual{\sheaf{E}}}(\OO_X)
  = R^ip_{2*}(\dual{\sheaf{E}})$ vanishes for all $i$ except $i_0$. In other words,
  \begin{equation}\label{eq:fm-trans2}
    \Phi_{\dual{\sheaf{E}}}(\OO_X)
    \iso \sheaf{F}[-i_0]
  \end{equation}
  for some coherent sheaf $\sheaf{F}$. Since
  $\Phi_{\dual{\sheaf{E}}}(-)$ and $\Psi_{\sheaf{E}}(-)[g]$ are
  quasi-inverse functors \cite[Proposition 5.9]{huybrechts}, the isomorphisms \eqref{eq:fm-trans1} and
  \eqref{eq:fm-trans2} give
  \begin{align*}
    H^p(X, \sheaf{G}_n)
    &\iso \Ext^p_X(\OO_X, \sheaf{G}_n)\\
    &\iso \Hom_{D(X)}(\OO_X, \sheaf{G}_n[p])\\
    &\iso \Hom_{D(Y)}(\Phi_{\dual{\sheaf{E}}}(\OO_X), \sheaf{L}^{-n}[p])\\
    &\iso \Hom_{D(Y)}(\sheaf{F}[-i_0], \sheaf{L}^{-n}[p])\\
    &\iso \Ext^{p+i_0}_Y(\sheaf{F},
    \sheaf{L}^{-n})\\
    &\iso \dual{H^{g-p-i_0}(Y,
    \sheaf{F}\tensor\sheaf{L}^n)},
  \end{align*}
  using Serre duality in the last step. If $n$ is sufficiently large,
  the cohomology group in the last line vanishes if and only if
$p$ differs from
  \begin{equation*}
    g-i_0 = i(\sheaf{Q})-\dim K(\sheaf{Q}).
  \end{equation*}
  Thus we have proved that $H^p(X, \sheaf{G}_n)$ is nonzero if and
  only if $p$ has this value. On the one hand, this shows that
  \emph{if} $\sheaf{G}_n$ satisfies IT, then this $p$ is its index. On the
  other hand, it also shows that $\sheaf{G}_n$ is nondegenerate, so it
  satisfies IT by Proposition \ref{prop:wit}, and we are done.
\end{proof}

The second part of Theorem \ref{thm:main} follows:

\begin{corollary}\label{cor:ample}
  The vector bundle $\sheaf{G}_n$ is ample for $n$ sufficiently large if and
  only if the bundles $\res{\sheaf{E}}{X\times\{y\}}$ parametrized by
  $Y$ are nef (equivalently, have nef determinant).
\end{corollary}

\begin{proof}
  A line bundle is ample if and only if its is nondegenerate and has
  index $0$. The same holds for any simple semihomogeneous vector
  bundle, since both conditions ``ample'' and ``nondegenerate of index
  $0$'' can be tested on the determinant line bundle, by
Proposition
  \ref{prop:det} and Proposition \ref{prop:wit}. Hence, by the theorem,
  the simple semihomogeneous vector bundle $\sheaf{G}_n$ is ample if and only
  if $i(\sheaf{Q})=\dim K(\sheaf{Q})$, where $\sheaf{Q}$ is the
  determinant of $\res{\sheaf{E}}{X\times\{y\}}$. By Corollary
\ref{cor:nef} this is
  equivalent to $\sheaf{Q}$ being nef, which again is equivalent to 
  nefness of $\res{\sheaf{E}}{X\times\{y\}}$, by Proposition
\ref{prop:det}.
\end{proof}

\bibliographystyle{amsplain}
\bibliography{semihom}

\end{document}